\documentclass[12pt,leqno,a4paper]{amsart}
\usepackage{amssymb,enumerate}
\newcommand{\fA}{{\mathfrak{A}}}
\newcommand{\fS}{{\mathfrak{S}}}

\newcommand{\bG}{{\mathbf{G}}}

\newcommand{\cE}{{\mathcal{E}}}

\newcommand{\Aut}{{\operatorname{Aut}}}
\newcommand{\Inn}{{\operatorname{Inn}}}
\newcommand{\Out}{{\operatorname{Out}}}
\newcommand{\Irr}{{\operatorname{Irr}}}
\newcommand{\fix}{{\operatorname{fix}}}

\newcommand{\PGL}{{\operatorname{PGL}}}
\newcommand{\PSL}{{\operatorname{L}}}

\newcommand{\PCSp}{{\operatorname{PCSp}}}
\newcommand{\PSp}{{\operatorname{S}}}
\newcommand{\OO}{{\operatorname{O}}}
\newcommand{\SO}{{\operatorname{SO}}}
\newcommand{\PCO}{{\operatorname{PCO}}}

\newcommand{\PGU}{{\operatorname{PGU}}}
\newcommand{\PSU}{{\operatorname{U}}}

\newcommand{\tw}[1]{{}^{#1}\!}

\newtheorem{thm}{Theorem}[section]
\newtheorem{lem}[thm]{Lemma}

\newtheorem{prop}[thm]{Proposition}

\theoremstyle{remark}

\begin{document}

\title[On the number of $p'$-degree characters]
      {On the number of $p'$-degree characters \\ in a finite group}

\date{\today}

\author{Gunter Malle and Attila Mar\'oti}
\address{FB Mathematik, TU Kaiserslautern, Postfach 3049,
         67653 Kaisers\-lautern, Germany.}
\email{malle@mathematik.uni-kl.de}
\email{maroti@mathematik.uni-kl.de}

\thanks{The first author gratefully acknowledges financial support by ERC
  Advanced Grant 291512. The second author was supported by an Alexander von
  Humboldt Fellowship for Experienced Researchers and by OTKA K84233.}

\keywords{height zero characters, McKay conjecture, simple groups}

\subjclass[2010]{20C15, 20C33}

\begin{abstract}
Let $p$ be a prime divisor of the order of a finite group $G$. Then $G$ has
at least $2 \sqrt{p-1}$ complex irreducible characters of degrees prime to $p$.
In case $p$ is a prime with $\sqrt{p-1}$ an integer this bound is sharp for
infinitely many groups $G$.
\end{abstract}

\maketitle


\section{Introduction}

Let $p$ be a prime and $G$ a finite group. Denote the set of complex
irreducible characters of $G$ whose degrees are prime to $p$ by $\Irr_{p'}(G)$.
The McKay Conjecture states that $|\Irr_{p'}(G)| = |\Irr_{p'}(N_{G}(P))|$
where $N_{G}(P)$ is the normalizer of a Sylow $p$-subgroup $P$ in $G$. Some 
known cases (easy consequence of \cite[Thm.~1]{Dade} and a special case of
\cite{IMN}) of this problem together with a recent result of the second author
\cite{Ma14} stating that the number of conjugacy classes in a finite group $G$
is at least $2 \sqrt{p-1}$ whenever $p$ is a prime divisor of the order of $G$
allows us to prove the following.

\begin{thm}   \label{thm:main}
 Let $G$ be a finite group and $p$ a prime divisor of the order of $G$. Then
 $|\Irr_{p'}(G)| \geq 2 \sqrt{p-1}$.
\end{thm}

Our proof of Theorem \ref{thm:main} shows that $|\Irr_{p'}(G)|$ is smallest
possible for a finite group $G$ whose order is divisible by a prime $p$ if
and only if the normalizer of a Sylow $p$-subgroup of $G$ has a certain special
structure. This may be natural in view of the (unsolved) McKay Conjecture.
Our second theorem gives a complete description of finite groups $G$ with the
property that $|\Irr_{p'}(G)| = 2 \sqrt{p-1}$ for a prime divisor $p$ of the
order of $G$, consistent with the McKay conjecture.

\begin{thm}   \label{thm:equality}
 Let $G$ be a finite group, $p$ a prime divisor of the order of $G$, and $P$
 a Sylow $p$-subgroup of $G$. Suppose that $\sqrt{p-1}$ is an integer and set
 $H$ to be the Frobenius group $C_{p} \rtimes C_{\sqrt{p-1}}$ (whose subgroup
 of order $p$ is self centralizing). Then $|\Irr_{p'}(G)| = 2\sqrt{p-1}$ if
 and only if $N_{G}(P) \cong H$.  \par
 Moreover this happens if and only if $G \cong H$, or $O_{p'}(G) = F(G)$, the
 subgroup $F(G)P$ is a Frobenius group, and $G/F(G)$ is either isomorphic to
 $H$ or is an almost simple group $A$ as described below.
\begin{enumerate}
  \item[\rm(1)] $p=5$ and $A=\fA_5$, $\fA_6$, $\PSL_2(11)$ or $\PSL_3(4)$;
  \item[\rm(2)] $p=17$ and $A=\PSp_4(4)$, $\OO_8^-(2)$ or $\PSL_2(16).2$;
  \item[\rm(3)] $p=37$ and $A={}^2G_2(27)$ or $\PSU_3(11).2$;
  \item[\rm(4)] $p=257$ and $A=\PSp_{16}(2)$, $\OO_{18}^-(2)$, $\PSL_2(256).8$, $\PSp_{4}(16).4$,
   $\PSp_{8}(4).2$, $\OO_8^-(4).4$, $\OO_{16}^-(2).2$ or $F_4(4).2$.
\end{enumerate}
\end{thm}

In Proposition~\ref{prop:e2} we show that for any prime $p$ with $\sqrt{p-1}$
an integer there are in fact infinitely many finite solvable groups $G$ with
$|\Irr_{p'}(G)| = 2\sqrt{p-1}$. We remark that it is an open problem first
posed by Landau whether there are infinitely many primes $p$ with $\sqrt{p-1}$
an integer (see e.g.~\cite[Sec.~19]{Pintz}).

\section{The McKay Conjecture}

Let $G$ be a finite group and $p$ a prime. The McKay Conjecture claims that
$|\Irr_{p'}(G)| = |\Irr_{p'}(N_{G}(P))|$ where $N_{G}(P)$ is the normalizer
of a Sylow $p$-subgroup $P$ in $G$. Thus if we wish to bound $|\Irr_{p'}(G)|$
and assume the validity of the McKay Conjecture for $G$ and $p$, then we may
assume that the Sylow $p$-subgroup $P$ is normal in $G$. In this case we have
$|\Irr_{p'}(G)| \geq |\Irr_{p'}(G/\Phi(P))|$ where $\Phi(P)$ is the Frattini
subgroup in $P$, a normal subgroup of $G$. Since $P/\Phi(P)$ is an elementary
abelian normal subgroup in $G/\Phi(P)$ which is also the Sylow $p$-subgroup
of $G/\Phi(P)$, by Clifford theory we have that all complex irreducible
characters of $G/\Phi(P)$ have degrees prime to $p$. But the number of
conjugacy classes of $G/\Phi(P)$ is at least $2 \sqrt{p-1}$ by
\cite[Thm.~1.1]{Ma14} with equality if and only if $\sqrt{p-1}$ is an integer
and $G/\Phi(P)$ is the Frobenius group $C_{p} \rtimes C_{\sqrt{p-1}}$ (whose
subgroup of order $p$ is self centralizing).

Now let us suppose that the McKay Conjecture is true for a finite group $G$
and a prime $p$. Then $|\Irr_{p'}(G)| = 2 \sqrt{p-1}$ if and only if the same
holds in case $G$ contains a normal Sylow $p$-subgroup $P$. By the previous
paragraph, $|P/\Phi(P)| = p$ so $P$ is cyclic. But then, by Clifford theory
once again, all complex irreducible characters of $G$ have degrees prime
to~$p$. Finally, by \cite[Thm. 1.1]{Ma14}, the number of conjugacy classes of
$G$ is equal to $2 \sqrt{p-1}$ if and only if $G$ is the Frobenius group
$C_{p} \rtimes C_{\sqrt{p-1}}$.

By the previous two paragraphs we showed Theorem~\ref{thm:main} and the first
half of Theorem \ref{thm:equality} in case the McKay Conjecture is true for
the pair $G$ and $p$. The McKay Conjecture is known to be true, for example,
for groups with a cyclic Sylow $p$-subgroup, by Dade \cite[Thm.~1]{Dade}.

\section{Reduction}
In this section we prove a reduction of Theorem \ref{thm:main} and of the
first half of Theorem \ref{thm:equality} to a question on finite non-abelian
simple groups.

Let $G$ be a finite group and $p$ a prime dividing the order of $G$. By the
previous section we can assume that the
Sylow $p$-subgroups of $G$ are not cyclic. So we would like to show
$|\Irr_{p'}(G)| > 2 \sqrt{p-1}$ in all remaining cases.

From the well-known identity $|G| = \sum_{\chi \in \Irr(G)} {\chi(1)}^{2}$
we see that $|\Irr_{p'}(G)| > 2 \sqrt{p-1}$ is true for $p=2$ and $p=3$.
So assume from now on that $p \geq 5$.

\subsection{Reduction to the monolithic case}

Let $G$ be a minimal counterexample to the bound, that is,
$|\Irr_{p'}(G)| \leq 2 \sqrt{p-1}$ and $G$ does not
have a cyclic Sylow $p$-subgroup.

Let $N$ be a minimal normal subgroup in $G$. Suppose first that $|G/N|$ is
divisible by $p$. Then $|\Irr_{p'}(G)| \geq |\Irr_{p'}(G/N)| \geq 2 \sqrt{p-1}$
by the minimality of $G$. So both inequalities must be equalities. But then
$G/N$ has a Sylow $p$-subgroup of order $p$ and $p^{2}$ divides
$$\sum_{\chi \in \Irr(G) \setminus \Irr(G/N)} {\chi(1)}^{2} = |G| - |G/N|.$$
This implies that $p^2$ cannot divide $|G|$ (only $p$). But we excluded the
case when $G$ has a cyclic Sylow $p$-subgroup.

So we must have that $|G/N|$ is not divisible by $p$, whence $|N|$ is
divisible by $p$. Then $N$ is an elementary abelian $p$-group or is a direct
product of simple groups $S$ having order divisible by $p$. By this argument
it also follows that $N$ is the unique minimal normal subgroup of $G$. If $N$
is abelian then $\Irr_{p'}(G) = \Irr(G)$ by Clifford theory and so we get the
result by \cite[Thm.~1.1]{Ma14}.

Thus $N = S_1 \times \cdots \times S_{t}$ where all $S_{i}$'s are isomorphic
to a non-abelian simple group $S$ having order divisible by $p$. Note that
$G/N$ permutes the simple factors transitively (but not necessarily faithfully).

\subsection{Reduction to simple groups}

We continue the investigation of a minimal counterexample $G$ as in the
previous subsection.
If $\psi \in \Irr_{p'}(N)$ then any irreducible character of $G$ lying above
$\psi$ has $p'$-degree by Clifford theory.

We wish to give a lower bound for the number of $G/N$-orbits on the set
$\Irr_{p'}(N)$. For this we may assume that $G/N$ is as large as possible,
subject to our conditions. So we may assume that $G = A \wr T$ where
$\Inn(S) \leq A \leq \Aut(S)$ is a group for which $|A/\Inn(S)|$ is
prime to $p$ and $T$ is a transitive permutation group on $t$ letters with
$|T|$ coprime to $p$ (but we may and will take $T$ to be $\fS_t$). Let $A_1$
be the stabilizer of $S_1$ in $G$. Let $K_1$ be the normal subgroup of
$A_1$ consisting of those elements which induce inner automorphisms on
$S_1$. Then $A_1/K_1$ can be considered as a $p'$-subgroup of
$\Out(S_1)$. Let $k$ be the number of $A_1$-orbits on $\Irr_{p'}(S_1)$.
Then $|\Irr_{p'}(G)| \geq \binom{k+t-1}{t}$.

Suppose for a moment that $t \geq 2$. Then $|\Irr_{p'}(G)| \geq \binom{k+1}{2}
= k(k+1)/2$. We want this to be larger than $2\sqrt{p-1}$. This is certainly
true if $k \geq 2 {(p-1)}^{1/4}$. On the other hand for $t=1$ we have $G = A$
and so we need $|\Irr_{p'}(G)| > 2\sqrt{p-1}$.

Thus Theorem \ref{thm:main} and the first part of Theorem \ref{thm:equality} is a consequence of the following result.

\begin{thm}   \label{thm:q}
 Let $S$ be a finite non-abelian simple group whose order is divisible by a
 prime $p$ at least $5$. Suppose that $S$ is not isomorphic to a projective
 special linear
 group $\PSL_2(q)$, a Suzuki group $\tw2B_2(q^2)$ or a Ree group $^2G_2(q^2)$.
 Let $X \leq \Aut(S)$ be a group containing $\Inn(S)$ so that $|X/\Inn(S)|$
 is not divisible by $p$. Furthermore let $k$ be the number of $X$-orbits on
 $\Irr_{p'}(S)$. Then
 \begin{enumerate}
  \item[\rm(a)] $k \ge 2 {(p-1)}^{1/4}$; and
  \item[\rm(b)] if the Sylow $p$-subgroups of $X$ are not cyclic then
   $|\Irr_{p'}(X)| > 2 \sqrt{p-1}$.
\end{enumerate}
\end{thm}		

Note that we may exclude the rank~1 groups $\PSL_2(q)$, $\tw2B_2(q^2)$ and
$^2G_2(q^2)$ in Theorem~\ref{thm:q}. Indeed, by Theorems~A and B and by the
comments in between on page~35 of \cite{IMN}, we see that the McKay Conjecture
is true for any corresponding $G$. So we may as well assume that $S$ is
different from these groups.

Note that if $X$ is as in Theorem~\ref{thm:q} then it is sufficient (but not
necessary) to show that $|\Irr_{p'}(X)| > 2 \sqrt{p-1} \cdot |X/S|$.

\section{Alternating and sporadic simple groups}
The aim of this section is to prove Theorem~\ref{thm:q} for alternating and
sporadic groups.

\subsection{The case when $S = \fA_n$}

Let us exclude the case $n=6$ from the discussion below because in this case
the full automorphism group of $S$ is not $\fS_n$.

We begin with a result of Macdonald (the following form of which
can be found in a paper by Olsson \cite{Olsson}). For a non-negative integer
$m$ let $\pi(m)$ denote the number of partitions of $m$. An $m$-split of a
non-negative integer $s$ is a sequence of non-negative integers
$(s_1, \ldots, s_{m})$ so that $\sum_{i=1}^{m} s_{i} = s$. Put
$k(m,s) = \sum \pi(s_1) \pi(s_{2}) \cdots \pi(s_{m})$ where the sum is over
all $m$-splits of $s$. (Notice that $k(m,0) = 1$.) For a prime divisor $p$
of $|\fS_n|$ let the $p$-adic expansion of the integer $n$ be
$a_{0} + a_1p + \cdots + a_{r}p^{r}$. Then Macdonald's result states that
$$|\Irr_{p'}(\fS_n)| = k(1,a_{0}) k(p,a_1) \cdots k(p^{r},a_{r}).$$
Notice that $m \cdot s \leq k(m,s)$ for all $m$ and $s$. This gives
$p - 1 \leq n- 1 \leq |\Irr_{p'}(\fS_n)|$ since the product of integers each
at least $2$ is always at least their sum. Thus
$$|\Irr_{p'}(\fA_n)| \geq k \geq (n-1)/2 \geq (p-1)/2.$$
A simple calculation shows that this is larger than $2 \sqrt{p-1}$ unless
$p \leq 17$. So we may assume that $5 \leq p \leq 17$, otherwise we are done.
But the same calculation can be applied using $n$ in place of $p$. So we may
also assume that $n \leq 17$.

If $a_{0} \geq 3$ or if $a_1 \geq 2$ or if $a_{i} \geq 1$ for some $i\geq 2$,
then $|\Irr_{p'}(\fS_n)| \geq 3p$. Using this bound and the calculation
referred to in the previous paragraph we get an affirmative answer to the
problem. So only the following cases are to be considered.

\begin{enumerate}
\item $n= p = 5$, $7$, $11$, $13$, $17$. In this case $|\Irr_{p'}(\fS_n)| = p$.

\item $n= p+1 = 8$, $12$, $14$. In this case $|\Irr_{p'}(\fS_n)| = p$.

\item $n= p+2 = 7$, $9$, $13$, $15$. In this case $|\Irr_{p'}(\fS_n)| = 2p$.
\end{enumerate}
	
For all the above values of $n$ and $p$ still to be considered (even for $n=6$)
we have that a Sylow $p$-subgroup of $X$ has order $p$, that is, is cyclic.
So we only have to bound $k$. 		
		
In the exceptional cases (1)--(3) above we certainly have $k \geq (p+1)/2$
since $p$ is odd. But then the bound in~(a) of Theorem~\ref{thm:q}
holds for $p \geq 5$.

Now suppose that $n = 6$. It is sufficient to
show in this case that $k \geq 2 {(p-1)}^{1/4}$ (where $p$ here is $5$). Since
the complex irreducible character degrees of $\fA_6$ are $1$, $5$, $5$, $8$,
$8$, $9$, $10$, we certainly have $k \geq 3$. But $3$ is larger than our
proposed bound.

\subsection{The case when $S$ is sporadic}

For sporadic groups and $\tw2F_4(2)'$ it is straightforward to check the
validity of the conditions in Theorem~\ref{thm:q} from the known character
tables in \cite{Atl}.

\section{Groups of Lie type}

Here, we prove Theorem~\ref{thm:q} for groups of Lie type. Let $G=\bG^F$ be
the group of fixed points under a Steinberg endomorphism $F$ of a simple
algebraic group $\bG$ of adjoint type over an algebraically closed field of
characteristic $r$. Let $p$ be a prime (which may coincide with $r$)
dividing $|G|$. Let $S$ be the simple socle of $G$.

\subsection{Two easy observations}

As above, $G$ is a finite reductive group of adjoint type.

\begin{lem}   \label{lem:adjoint}
 Suppose that $p$ does not divide $|G/S|$. Then the
 claim of Theorem~\ref{thm:q} holds for $(S,p)$ if
 $2 \sqrt{p-1} \cdot |\Out(S)|_{p'} < |\Irr_{p'}(G)|$.
\end{lem}

\begin{proof}
By the condition on $G$, by Schreier's conjecture, and by Hall's theorem, we
may assume that $X$ contains $G$. Now
$2 \sqrt{p-1} \cdot |\Out(S)|_{p'} < |\Irr_{p'}(G)|$ implies that
$2 \sqrt{p-1} \cdot |X/S| < |\Irr_{p'}(G)|$. From this we have
$$2 \sqrt{p-1} < \frac{|G|}{|X|} \cdot \frac{|\Irr_{p'}(G)|}{|G:S|}
  \leq \frac{|G|}{|X|} \cdot \Big( \frac{1}{|G|} \sum_{g \in G} |\fix(g)|\Big)
  \leq \frac{1}{|X|} \sum_{g \in X} |\fix(g)| = k$$
where $|\fix(g)|$ denotes the number of fixed points of $g \in X$ on
$\Irr_{p'}(S)$.
\end{proof}

Here is a further easy sufficient criterion:

\begin{lem}   \label{lem:easy}
 Let $S$ be non-abelian simple. Assume that there is $I\subseteq\Irr_{p'}(S)$
 such that all $\chi\in I$ are $\Out(S)$-invariant and extend to $\Aut(S)$.
 Then the conclusion of Theorem~\ref{thm:q} holds for $(S,p)$ if one of the
 following conditions holds:
 \begin{enumerate}
 \item[\rm(1)] $p<|I|^2/4+1$, or
 \item[\rm(2)] Sylow $p$-subgroups of $\Aut(S)$ are cyclic and $p\le|I|^4/16+1$.
 \end{enumerate}
\end{lem}

\begin{proof}
By assumption $\Out(S)$ has at least $k:=|I|$ orbits on $\Irr_{p'}(S)$. Since
all characters of $I$ extend to $\Aut(S)$, any $S\le X\le\Aut(S)$
has $|\Irr_{p'}(X)|\ge k$. Now $k=|I|>2(p-1)^{1/2}\ge 2(p-1)^{1/4}$, so $(S,p)$
satisfies the condition in Theorem~\ref{thm:q}(b). If Sylow $p$-subgroups of
$\Aut(S)$ are cyclic, we just need $k>2(p-1)^{1/4}$.
\end{proof}

Note that for invariant characters extendibility to $\Aut(S)$ is automatically
satisfied if all Sylow subgroups of $\Out(S)$ are cyclic, for example.

\subsection{The defining characteristic case (for rank $l \geq 2$)}

\begin{prop}
 Theorem~\ref{thm:q} holds for $S$ of Lie type in characteristic~$p$.
\end{prop}

\begin{proof}
As before, let $\bG$ be a simple linear algebraic group in characteristic~$p$
of adjoint type with
a Steinberg endomorphism $F:\bG\rightarrow\bG$ and $G:=\bG^F$ such that
$S=[G,G]$. All finite simple groups of Lie type are of this form (see
\cite[Prop.~24.21]{MT}). We denote by $(\bG^{*},F^{*})$ the dual pair of
$(\bG,F)$ (see \cite[Sec.~4.2]{Ca}). Here $\bG^{*}$ is a simple
algebraic group of simply connected type. We denote the corresponding finite
group of Lie type by $G^*$. By \cite[Prop.~24.21]{MT}, we have
$G^*/Z(G^*) \cong [G,G]=S$. Since
$p \geq 5$, we know by \cite[Lemma 5]{Brunat} that the set of $p'$-degree
complex irreducible characters of $G$ is precisely the set of
semisimple characters of $G$, whose elements are labelled by
representatives of the conjugacy classes of semisimple elements of
$G^*$. Thus $|\Irr_{p'}(G)| = q^{l}$ where $l$
is the semisimple rank of $\bG^*$, and $q$ is the absolute value of all
eigenvalues of $F$ on the character group of an $F$-stable maximal torus of
$\bG$, by \cite[Thm.~3.7.6(ii)]{Ca}.

By Clifford theory we then have
$$q^{l} = |\Irr_{p'}(G)| \leq |G:S| \cdot t$$
where $t$ is the number of $G/S$-orbits on $\Irr_{p'}(S)$. By the
orbit-counting lemma,
$$q^{l} \leq |G:S| \cdot t = \sum_{g \in G/S} |\fix(g)|
   \leq \sum_{g \in \Out(S)} |\fix(g)| \leq k \cdot |\Out(S)|.$$
So we get $q^{l}/|\Out(S)| \leq k$.

In order to prove Theorem~\ref{thm:q} for $(S,p)$ it is sufficient to see that
$q^{l}/|\Out(S)| > 2 \sqrt{p-1}$, where $q=p^f$. Bounds for $|\Out(S)|$ can be
read off from \cite[Tab.~5]{Atl}. If $(f,l,p) \ne (1,2,5)$ nor $(1,2,7)$, then
the bound $|\Out(S)| \leq (6 l + 3)f$ is sufficient for our purposes (note
that $l \geq 2$). On the other hand, if $(f, l,p) = (1,2,5)$ or
$(1,2,7)$ then the bounds $|\Out(S)| \leq 6$ and $|\Out(S)| \leq 8$ are
sufficient, respectively. 		
\end{proof}

\subsection{Exceptional type groups in non-defining characteristic}   \label{subsec:exc}

\begin{prop}   \label{prop:exc}
 Let $S$ be a simple exceptional group of Lie type, not of type $\tw2B_2$ or
 $^2G_2$, and $p\ge5$ a prime dividing $|S|$ but different from the defining
 characteristic. Then $(S,p)$ satisfies the conclusion of Theorem~\ref{thm:q}.
\end{prop}

\begin{proof}
Let $G$ be a finite reductive group of adjoint type with socle $S$. We first
deal with the primes $p$ for which Sylow $p$-subgroups of $G$ are non-abelian.
These necessarily divide the order of the Weyl group $W$ of $G$, so $p\le7$,
and $G$ is of type $\tw{(2)}E_6$, $E_7$ or $E_8$. Furthermore, $p|(q\pm1)$
if $p=7$, or if $p=5$ and $G$ is not of type $E_8$. It is then straightforward
to check (for example from the tables in \cite[\S13.9]{Ca}) that $G$ has at
least as many unipotent characters of $p'$-degree as given in
Table~\ref{tab:invunismall}. Since unipotent characters extend to $\Aut(S)$
by \cite[Thm.~2.5]{MaE}, the claim follows from Lemma~\ref{lem:easy} in this
case.

\begin{table}[htbp]
 \caption{Invariant unipotent characters, $p\in\{5,7\}$}  \label{tab:invunismall}
\[\begin{array}{|r|ccc|}
\hline
   G& \tw{(2)}E_6& E_7& E_8\cr
\hline
 p=5& 10& 30& 20\\
 p=7&  -& 14& 28\\
\hline
\end{array}\]
\end{table}

We may now assume that Sylow $p$-subgroups of $G$ are abelian. Then there
exists a unique cyclotomic polynomial $\Phi_d$ dividing the generic order of
$G$ and such that $p|\Phi_d(q)$. Moreover, there exists a maximal torus $T_d$
of $G$ containing a Sylow $d$-torus of $G$, and so in particular a Sylow
$p$-subgroup of $G$ (see \cite[Thm.~25.14]{MT}). Let $\Phi_d^{a_d}$ be the
precise power of $\Phi_d$ dividing the order polynomial of $G$. The Sylow
$p$-subgroups of $G$ are cyclic if and only if $a_d=1$. Let $W_d$ be the
relative Weyl group of $T_d$. Then by generalized Harish-Chandra theory (or
alternatively from the formulas in \cite[\S13.9]{Ca}) there
exist at least $|\Irr(W_d)|$ many unipotent characters of $G$ of $p'$-degree.
By \cite[Thms.~2.4 and~2.5]{MaE} all of these extend to $\Aut(S)$ unless $G$
is of type $G_2$ and $r=3$, or of type $F_4$ and $r=2$. The various $W_d$ and
$a_d$ are explicitly known (see e.g. \cite[Tables~1 and~3]{BMM}), and applying
Lemma~\ref{lem:easy} we conclude that our claim holds if $p$ is as in
Table~\ref{tab:invuni}. Here, the left-most half of the table contains the
cases with $a_d>1$, while in the right-most part we have $a_d=1$, so Sylow
$p$-subgroups are cyclic.

\begin{table}[htbp]
 \caption{$\Aut(S)$-invariant unipotent characters}  \label{tab:invuni}
\[\begin{array}{|cl|r|l||l|r|l|}
\hline
     G& d& \#& \quad p&   d& \#&\quad p\cr
\hline
     G_2& 1,2&  6& p\le10&    3,6& 6& p\le82\\
 \tw3D_4& 1,2&  6& p\le10&    12&  4& p\le17\\
        & 3,6&  7& p\le13&          & & \\
 \tw2F_4& 1,4,8',8''& 7& p\le13& 12,24',24''& 12& p\le1297\\
     F_4& 1,2& 11& p\le31&    8,12& \ge8& p\le257\\
        & 3,6& 9& p\le21&      & & \\
^{(2)}E_6& 1,2,3,4,6& \ge16& p\le65& 5,8,9,12,(10,18)& \ge5& p\le40\\
     E_7& 1,2,3,4,6& \ge48& p\le577& 5,7,8,9,10,12,14,18& \ge14& p\le2402\\
     E_8& 1,2,3,4,6& \ge59& p\le871& 7,9,14,18& \ge28& p\le38417\\
        & 5,8,10,12& \ge32& p\le257& 15,20,24,30& \ge20& p\le10001\\
\hline
\end{array}\]
\end{table}

So from now on we suppose that $p$ is larger than the bound given in the
table.
Let $d,T_d,W_d$ be as above. Let $s\in T_d$ be semisimple. Then $s$ centralizes
a Sylow $p$-subgroup of $G$, so the semisimple character in the Lusztig series
$\cE(G,s)$ has degree prime to~$p$ by Lusztig's Jordan decomposition
(see e.g.~\cite[Prop.~7.2]{MaH}). Since fusion of semisimple elements in
maximal tori is controlled by the relative Weyl group, there exist at least
$|T_d|/|W_d|$ semisimple conjugacy classes of $G$ with representatives in
$T_d$, whence $|\Irr_{p'}(G)|\ge |T_d|/|W_d|$. We now go through the various
types of groups.
\par
Let first $G=S=G_2(q)$ with $q=r^f>2$ (as $G_2(2)\cong\Aut(\PSU_3(3))$).
Then $\Out(S)$ is cyclic of order $f$ for $r\ne3$ respectively $2f$ for $r=3$,
and $d\in\{1,2,3,6\}$, with $a_d=2$ for $d=1,2$ and $a_d=1$ else.
Table~\ref{tab:invuni} then shows that $q\ge 11$. It is now straightforward
to check that $|T_d|/|W_d|>2\sqrt{p-1}|\Out(S)|$, so the condition in
Lemma~\ref{lem:adjoint} is satisfied in these cases.
\par
Next consider $G=S=\tw3D_4(q)$, $q=r^f$. As before, $\Out(S)$ is cyclic, of
order $3f$. Here, we have $d\in\{1,2,3,6,12\}$, with $a_d=2$ for $d\le6$.
By Table~\ref{tab:invuni} we may assume that $q\ge11$. In all cases the
estimate above gives the claim.
The same arguments also apply to $\tw2F_4(2^{2f+1})$ and $F_4(q)$.
\par
Now assume that $G=E_6(q)$, $q=r^f$. Here the outer automorphism group is of
order $2f\gcd(3,q-1)$, but no longer cyclic. We have
$d\in\{1,2,3,4,5,6,8,9,12\}$. First assume that Sylow
$p$-subgroups are cyclic, so $d\in\{5,8,9,12\}$. Then $p\ge41$ by
Table~\ref{tab:invuni}, and $|W_d|\le12$. The standard estimate now applies.
For $d\in\{2,3,4,6\}$ we have $67\le p\le q^2+1$, while $|T_d|\ge (q^2-q)^3$
and $|W_d|\le 1152$, while for $d=1$ we have $67\le p\le q-1$ and
$|T_d|=(q-1)^6$. In all cases we obtain a contradiction to the standard
estimate. The case of $\tw2E_6(q)$ can be handled similarly.
For $E_7(q)$ the outer automorphism group has order $f\gcd(2,q-1)$, and
the same approach as before applies.
Finally, let $G=S=E_8(q)$ with $q=r^f$. Then $|\Out(S)|=f$.
We now discuss the various possibilities for $d$. If $d=1$, so $p|(q-1)$,
then $W_d$ is the Weyl group of $G$, with $|\Irr(W_d)|=112$. So we are done
whenever $2f\sqrt{p-1}<112$, which certainly is the case for $q\le1000$. For
$q\ge1001$ we have
$$\Phi_d(q)^a/|W_d|=(q-1)^8/696729600>2\log_p(q)\sqrt{p-1}.$$
The case $d=2$ is very similar. For $d=3$ or $d=6$, $|W_d|=155\,520$ (see
\cite[Table~3]{BMM}) and $|\Irr(W_d)|=102$. We may conclude as before.
Similarly, for $d=4$ we have $|W_d|=46080$ and $|\Irr(W_d)|=59$;
for $d=5$ or $d=10$ we have $|W_d|=600$ and $|\Irr(W_d)|=45$; for $d=12$ we
have $|W_d|=288$ and $|\Irr(W_d)|=48$. Finally, for the cases
$d\in\{7,14,9,18,15,20,24,30\}$ with cyclic Sylow $p$-subgroups the estimates
are even easier, using the bounds in Table~\ref{tab:invuni}.
This achieves the proof.
\end{proof}

\subsection{Groups of classical type in non-defining characteristic}   \label{subsec:class}

\begin{prop}   \label{prop:class}
 Let $S$ be a simple classical group of Lie type and $p\ge5$ a prime dividing
 $|S|$ but different from the defining characteristic. Then $(S,p)$ satisfies
 the conclusion of Theorem~\ref{thm:q}.
\end{prop}

\begin{proof}
Let first $G=\SO_{2n+1}(q)$ or $\PCSp_{2n}(q)$ with $q=r^f$ and $n\ge2$.
Here $\Out(S)$ is cyclic of order~$f\gcd(2,q-1)$, respectively of order $2f$
if $n=2$ and $q$ is even. Let $d$ be minimal such that $p$ divides $q^d\pm1$.
A Sylow $d$-torus $T_d$ of $G$ has order $\Phi_d^a$ when $n=ad+s$ with
$0\le s<d$. The centralizer of $T_d$ in $G$ has a subgroup of the form
$(q^d\pm1)^a G_s(q)$, where $G_s$
has the same type as $G$ and rank~$s$ (see \cite[\S3A]{BMM}). The relative Weyl
group $W_d$ of $T_d$ is the wreath product $C_{2d}\wr \fS_a$.\par
If Sylow $p$-subgroups of $G$ are non-abelian, then $p\le n$ divides $|W_d|$,
whence $p\le a$ as $p$ cannot divide $d$. Now the number of unipotent
characters of $p'$-degree of $G$ in the principal $p$-block is at least the
number of $p'$-characters of $W_d$, hence of its factor group $\fS_a$, hence
at least $p-1$, and all of these are $\Out(S)$-invariant by
\cite[Thm.~2.5]{MaE}, so we are done in this case.
\par
Else, the centralizer of $T_d$ contains a Sylow $p$-subgroup of $G$, whence all
semisimple elements of the torus of order $(q^d\pm1)^a$ give rise to
semisimple characters of $G$ in $\Irr_{p'}(G)$, and in addition the unipotent
characters in the principal $p$-block of $G$, of which there are
$|\Irr(W_d)|$ many, have degree coprime to $p$. Thus by Lemma~\ref{lem:adjoint}
if suffices to show that
$$|\Irr(W_d)|+\frac{(q^d-1)^a}{(2d)^a\, a!} > 2f\gcd(2,q-1)\sqrt{p-1}$$
where $p|(q^d\pm1)$. If $a=1$ then Sylow $p$-subgroups of $\Aut(G)$ are cyclic.
Otherwise it is easily seen that this inequality always holds.
\par
Next let $G=\PCO_{2n}^\pm(q)$ with $q=r^f$ and $n\ge4$. Here $\Out(S)$ has
order~$fg\gcd(4,q^n\pm1)$, where $g=6$ for $n=4$ and $g=2$ else denotes the
number of graph automorphisms. Let again $d$ be minimal such that $p$ divides
$q^d\pm1$. The situation is very similar to the one for groups of types $B_n$
and $C_n$, except that the relative Weyl group $W_d$ sometimes is a subgroup
of index two in the wreath product $C_{2d}\wr\fS_a$. Arguing as before we
find that there are no cases with $a>1$ violating the above inequality. For
$a=1$ Sylow $p$-subgroups of $G$ are cyclic.
\par
Next let $G=\PGL_n(q)$ with $q=r^f$ and $n\ge3$. Let $d$ be minimal with
$p$ dividing $q^d-1$ and write $n=ad+s$ with $0\le s<d$. A Sylow $d$-torus
$T_d$ of $G$ has order $\Phi_d^a$. The centralizer of $T_d$ in $G$ contains a
subgroup of the form $(q^d-1)^a G_s(q)$, where $G_s$ is of type $A_{s-1}$.
The relative Weyl group $W_d$ of $T_d$ is the wreath product $C_{d}\wr \fS_a$.
\par
If Sylow $p$-subgroups of $G$ are non-abelian, then $p\le n$ divides $|W_d|$,
and so $p\le a$. Again, the number of unipotent characters of $p'$-degree of
$G$ in the principal $p$-block is at least the number of $p'$-characters of
$W_d$, hence of $\fS_a$, hence at least $p-1$. Since all of these are
$\Out(S)$-invariant, we are done in this case.
\par
Otherwise we may assume that $a>1$. Arguing as in the case of the other
classical groups, we arrive at the following inequality
$$|\Irr(W_d)|+\frac{(q^d-1)^a}{d^a\, a!} > 2f\gcd(n,q-1)\sqrt{p-1},$$
which turns out to be satisfied for all relevant values.
\par
The case of $G=\PGU_n(q)$ is entirely similar, which $q^d-1$ replaced by
$q^d-(-1)^d$ throughout. The proof is complete.
\end{proof}

\section{Proof of Theorem \ref{thm:equality}}

In this section we prove Theorem~\ref{thm:equality}.

\begin{lem}   \label{e1}
 Let $G$ be a finite group, $p$ a prime divisor of the order of $G$, and $P$
 a Sylow $p$-subgroup of $G$. Suppose that $\sqrt{p-1}$ is an integer and set
 $H$ to be the Frobenius group $C_{p} \rtimes C_{\sqrt{p-1}}$ (whose subgroup
 of order $p$ is self centralizing).
 Then $|\Irr_{p'}(G)| = 2\sqrt{p-1}$ if and only if $N_{G}(P) \cong H$.
 Moreover this happens if and only if $G \cong H$, or $O_{p'}(G) = F(G)$, the
 subgroup $F(G)P$ is a Frobenius group, and $G/F(G)$ is either isomorphic
 to $H$ or is an almost simple group $A$ with $N_{A}(F(G)P/F(G)) \cong H$.
\end{lem}

\begin{proof}
We have already proved the first statement of the lemma in the preceding
sections.

So now suppose that $N_{G}(P) \cong H$ holds. Then by Theorem~\ref{thm:main},
we have
$$2 \sqrt{p-1}\leq |\Irr_{p'}(G/O_{p'}(G))| \leq |\Irr_{p'}(G)| = 2\sqrt{p-1}$$
and so $N_{G/O_{p'}(G)}(Q) \cong H$ for a Sylow $p$-subgroup $Q$ of
$G/O_{p'}(G)$. Since $O_{p'}(G/O_{p'}(G)) = 1$ and $|Q|=p$, we see that either
$Q$ is normal in $G/O_{p'}(G)$ and thus $G/O_{p'}(G) \cong H$, or $G/O_{p'}(G)$
is almost simple. Since $P$ is self centralizing in $G$, it acts fixed point
freely on $O_{p'}(G)$ and so $O_{p'}(G) P$ is a Frobenius group. By Thompson's
theorem \cite[Thm.~5.1']{W}, $O_{p'}(G) \leq F(G)$. The other containment
follows from $P \not\leq F(G)$ whenever $G \not\cong H$.

Now consider the other implication of the second statement of the lemma.
Assume that $G \not\cong H$. Since $F(G)P$ is a Frobenius group, we have
$N_{G}(P) \cap F(G) = 1$. Furthermore $N_{G}(P)$ is isomorphic to
$N_{G/F(G)}(F(G)P/F(G)) \cong H$.
\end{proof}

To finish the proof of Theorem \ref{thm:equality}, we need to classify almost
simple groups $A$ with the property that the normalizer of a Sylow
$p$-subgroup in $A$ is the Frobenius group $C_{p} \rtimes C_{\sqrt{p-1}}$
(whose subgroup of order $p$ is self centralizing).

\begin{prop}   \label{prop:almost simple}
 Let $A$ be a finite almost simple group and $p$ a prime. Then the
 Sylow $p$-subgroups of $A$ are as described in Lemma~\ref{e1}
 if and only if $A$ is as in (1)--(4) of Theorem \ref{thm:equality}.
\end{prop}

\begin{proof}
Note that the smallest primes $p>2$ such that $\sqrt{p-1}$ is an integer
are given by $5,17,37,101,197,257,...$ Assume that $A$ is a non-abelian almost
simple group with socle $S$ and with a Sylow $p$-subgroup as in
Theorem~\ref{thm:equality}. For $S$ a sporadic group, it is readily checked
from the Atlas \cite{Atl} that no example arises (only the primes $p=5,17,37$
are relevant). Now let $S=\fA_n$ with $n\ge5$. Any element of $\fS_n$ is
rational, so any element of order $p$ of $\fA_n$ is conjugate to at least
$(p-1)/2$ of its powers. But $(p-1)/2\le\sqrt{p-1}$ if and only if $p=5$,
and 5-cycles are non-rational only in $\fA_5$ and in $\fA_6$. This occurs in
exception~(1). \par
If $S$ is of Lie type in defining characteristic, its Sylow $p$-subgroups have
order~$p$ only when $S=\PSL_2(p)$, in which case the automizer has order
$(p-1)/\gcd(p-1,2)$. Again, only $p=5$ and $A=\PSL_2(5)=\fA_5$ arises.
\par
Now assume that $S$ is of Lie type but $p$ is not the defining characteristic.
Note that if $p$ divides $|A|$, then it divides $|S|$, unless $A$ contains
a coprime field automorphism. But the latter have non-trivial centralizer in
$S$, so indeed we may suppose that $p$ divides $|S|$.
If $p$ divides the order of the Weyl group of $S$, then $p^2$ divides $|S|$,
so this is not the case. Otherwise Sylow $p$-subgroups of $S$ are abelian and
contained in some maximal torus $T$ of $S$. In particular this torus must
be of prime order~$p$ and self-centralizing. Let $m:=|N_A(T)/T|$, then moreover
$m^2+1=|T|=p$.
So in particular $m$ has to be even. First assume that $S$ is of exceptional
Lie type. It is easily seen that under the above restrictions the only example
is $^2G_2(27)$ with $p=37$ as in~(3), or $F_4(4).2$ with $p=257$ as in~(4).
For example, for $A=E_8(q)$, $q=r^f$, the only possible values for $m$ are
$m=15u,20u,24u,30u$ where $u|f$, while $|T|\ge q^8-q^7+q^5-q^4+q^3-q+1$ for
cyclic maximal tori, which clearly gives no example.
\par
Finally we handle the case that $A$ is of classical Lie type. If $A$ is of
type $B_n(q)$ or $C_n(q)$ with $n\ge2$ the only cyclic self-centralizing tori
have order $(q^n\pm1)/\gcd(2,q-1)$ and automizer of order $2nf$, where $q=r^f$.
But $(q^n\pm1)/\gcd(2,q-1)=(2n)^2+1$ only has the solutions given in cases~(2)
and (4).
For $A$ of type $D_n(q)$ with $n\ge4$ the cyclic self-centralizing tori are of
order $(q^n-1)/\gcd(4,q^n-1)$ with automizer of order $n$, and of order
$q^{n-1}-1$ with $q=2$ with automizer of order $2(n-1)$. These do not lead
to examples.
For groups of type $\tw2D_n(q)$ the cyclic self-centralizing tori are of order
$(q^n+1)/\gcd(2,q^n+1)$ with automizer of order $n$, and of order $q^{n-1}+1$
with $q=2$ with automizer of order $2(n-1)$. The only examples here are those
in~(2) and (4).
\par
Now assume that $S=\PSL_n(q)$ with $n\ge2$. Here, cyclic self-centralizing
tori have orders $(q^n-1)/(q-1)/d$ with automizer of order $n$, and
$(q^{n-1}-1)/d$ with automizer of order $n-1$, where $d:=\gcd(n,q-1)$. This
leads to $\PSL_2(4)\cong\fA_5$, $\PSL_2(9)\cong\fA_6$, $\PSL_2(11)$,
$\PSL_3(4)$, $\PSL_{2}(16).2$ and $\PSL_{2}(256).8$. Finally, for
unitary groups $S=\PSU_n(q)$ with $n\ge3$, cyclic self-centralizing tori have
orders $(q^n-(-1)^n)/(q+1)/d$ with automizer of order $n$, and
$(q^{n-1}-(-1)^{n-1})/d$ with automizer of order $n-1$, where $d:=\gcd(n,q+1)$.
This gives $(A,p) = (\PSU_{3}(11).2,37)$ as the only example.
\end{proof}

Finally we prove the last statement of the Introduction.

\begin{prop}   \label{prop:e2}
 For any prime $p$ with $\sqrt{p-1}$ an integer there are infinitely many
 finite solvable groups $G$ with $|\Irr_{p'}(G)| = 2\sqrt{p-1}$.
\end{prop}

\begin{proof}
By Dirichlet's theorem on arithmetic progressions there are infinitely many
primes $r$ of the form $pn + 1$ where $n$ is an integer. Pick such an $r$
and set $m:=\sqrt{p-1}$. Let $V$ be an $m$-dimensional vector space over the
field with $r$ elements. Then $\Gamma$L$(V)$ contains a subgroup
$\Gamma$L$_1(r^m) \cong C_{r^m-1} \rtimes C_m$. Since $p$ divides
$r^m-1$, this former group contains a (unique) subgroup $A$ of the form
$C_{p} \rtimes C_m$. We claim that $C_{A}(P) = P$ where $P$ is the Sylow
$p$-subgroup of $A$. Let $x$ be a generator of $P$ and let $y$ be a generator
of a cyclic subgroup of order $m$ in $A$ so that $x^y = x^r$. We have to show
that whenever $s$ is an integer with $1 \leq s < m$, then $x^{r^s} \not= x$.
But this is clear since $r^m-1$ does not divide $r^{s}-1$.

Now set $G = V\rtimes A$. Then $O_{p'}(G) = F(G) = V$, $VP$ is a Frobenius
group, and $G/V = A$ is a Frobenius group of the form $C_{p} \rtimes C_m$.
Now apply Lemma~\ref{e1}.
\end{proof}


\end{document}